\documentclass[12pt]{amsart}

\usepackage[OT2,OT1]{fontenc}
\usepackage[all]{xy}

\newcommand\cyrillic[1]{
    {\fontencoding{OT2}\fontfamily{wncyr}\selectfont #1}
        }

\newcommand\mathcyr[1]{\text{\cyrillic{#1}}}
\newcommand\Sha{\textnormal{\mathcyr{Sh}}} 

\usepackage{graphicx}

\usepackage{amsmath,amssymb,framed,graphicx,latexsym,youngtab,epstopdf,booktabs}
\usepackage{color,fourier}

\usepackage{array}
\usepackage{url}

\newcommand{\lbincomb}[5]{\ensuremath{
\vcenter{\hbox{\xymatrix@R=.4pc@C=.2pc{ 
            #3\ar@{-}[dr] &&#4\ar@{-}[dl] &\\
            &*+[o][F-]{#2}\ar@{-}[dr]&&#5\ar@{-}[dl]\\
            &&*+[o][F-]{#1}\ar@{-}[d]&\\
            &&*{}&
}}}}}
    
\newcommand{\rbincomb}[5]{\ensuremath{
\vcenter{\hbox{\xymatrix@R=.4pc@C=.2pc{ 
             &#4\ar@{-}[dr] &&#5\ar@{-}[dl] \\
            #3\ar@{-}[dr]&&*+[o][F-]{#2}\ar@{-}[dl]\\
            &*+[o][F-]{#1}\ar@{-}[d]&&\\
            &*{}&&
}}}}}

\newtheorem{theorem}{Theorem}

\newtheorem{proposition}[theorem]{Proposition}

\theoremstyle{definition}

\newtheorem{remark}[theorem]{Remark}

\allowdisplaybreaks

\def\calF{\mathcal{F}}
\def\calG{\mathcal{G}}

\def\calP{\mathcal{P}}
\def\calQ{\mathcal{Q}}

\def\calS{\mathcal{S}}
\def\calT{\mathcal{T}}

\def\calX{\mathcal{X}}
\def\calY{\mathcal{Y}}

\DeclareMathOperator{\gr}{gr}

\DeclareMathOperator{\Refl}{\textsl{Reflex}}
\DeclareMathOperator{\Perm}{\textsl{Perm}}
\DeclareMathOperator{\CP}{\textsl{CycPerm}}
\DeclareMathOperator{\Lie}{\textsl{Lie}}
\DeclareMathOperator{\PL}{\textsl{PreLie}}
\DeclareMathOperator{\CL}{\textsl{CycLie}}
\DeclareMathOperator{\CPL}{\textsl{CycPreLie}}
\DeclareMathOperator{\RT}{\textsl{RT}}

\DeclareMathAlphabet{\mathbbold}{U}{bbold}{m}{n}

\def\k{\mathbbold{k}}

\begin{document}

\title{Fine structures inside the pre-Lie operad revisited}

\author{Vladimir Dotsenko}

\address{ 
Institut de Recherche Math\'ematique Avanc\'ee, UMR 7501, Universit\'e de Strasbourg et CNRS, 7 rue Ren\'e-Descartes, 67000 Strasbourg CEDEX, France}

\email{vdotsenko@unistra.fr}

\date{}

\begin{abstract}
We prove a conjecture of Chapoton from 2010 stating that the pre-Lie operad, as a Lie algebra in the symmetric monoidal category of linear species, is freely generated by the free operad on the species of cyclic Lie elements. 
\end{abstract}

\maketitle

\section*{Introduction}
Recall that a pre-Lie algebra is a vector space with a binary operation $a,b\mapsto a\triangleleft b$ satisfying the identity
 \[
(a_1\triangleleft a_2)\triangleleft a_3 - a_1\triangleleft (a_2\triangleleft a_3) = (a_1\triangleleft a_3)\triangleleft a_2- a_1\triangleleft (a_3\triangleleft a_2).
 \]
This is a remarkable algebraic structure appearing in many different areas of mathematics, including category theory, combinatorics, deformation theory, differential geometry, and mathematical physics, to name a few. It seems to have first appeared independently in the work of Gerstenhaber \cite{MR161898} and Vinberg \cite{MR0158414} in 1960s. In particular, in each pre-Lie algebra the commutator $[a_1,a_2]=a_1\triangleleft a_2-a_2\triangleleft a_1$ satisfies the Jacobi identity, and hence defines a Lie algebra structure on the same vector space.

A fact that could have been discovered by Cayley \cite{cayley_2009} in 1857, but had to wait till the work of Chapoton and Livernet \cite{MR1827084} in 2000, states that free pre-Lie algebras can be described using rooted trees. For the case of the free pre-Lie algebra on one generator, the corresponding commutator algebra, the Lie algebra of rooted trees, plays an important role in the celebrated construction of Connes and Kreimer \cite{MR1633004}. A result going back to the work of Foissy \cite{MR1935036} in 2001 states that the Lie algebra of rooted trees is a free Lie algebra; in fact, it is easy to generalize that proof to the case of the commutator algebra of any free pre-Lie algebra. In 2007, Chapoton~\cite{MR2682539} and Bergeron and Livernet \cite{MR2667815}, established the same result in a more functorial way, proving that the pre-Lie operad is free as a Lie algebra in the symmetric monoidal category of linear species. 

In 2010, Bergeron and Loday proved \cite{MR2763748} that in any pre-Lie algebra, the symmetrized product $a_1\bullet a_2=a_1\triangleleft a_2+a_2\triangleleft a_1$ does not satisfy any identity other than commutativity, which is a striking contrast with the symmetrized product in associative algebras \cite{MR186708}. Soon after their work was circulated, Chapoton proposed in~\cite{MR3091764} a beautiful conjecture substantially generalizing this result. To state that conjecture precisely, let us recall a definition of the species of cyclic Lie elements (whose components are known as the Whitehouse modules for symmetric groups \cite{MR1394355,Whitehouse}). In the context of this paper, it is convenient to pick the version of the definition describing it in the realm of modules over non-cyclic operads. Specifically, the species $\CL$ is the underlying species of the right $\Lie$-module generated by one element $(-,-)$ of arity $2$ satisfying the symmetry condition $(a_1,a_2)=(a_2,a_1)$ and the right $\Lie$-module relation  
 \[
([a_1,a_2],a_3)=(a_1,[a_2,a_3]) 
 \]
with the Lie bracket $[-,-]\in\Lie(2)$; intuitively, this corresponds to thinking of cyclic Lie elements as universal values of invariant bilinear forms on Lie algebras, see \cite[Sec.~4]{MR1358617}. The conjecture of Chapoton states that the species that freely generates the pre-Lie operad as a Lie algebra is isomorphic to the underlying species of the free operad generated by the species $\CL$. In this paper, a simple proof of that conjecture is given.  

\section{The main theorem}

All vector spaces and chain complexes are defined over a ground field~$\k$ of zero characteristic. For necessary information about operads, we refer the reader to \cite{MR2954392}; it is important to think of operads as monoids in the category of linear species, since that, in particular, greatly helps the intuition on operadic modules and bimodules upon which the below proof crucially relies. In particular, we shall repeatedly use the fact that a morphism of operads $\phi\colon\calP\to\calQ$ makes $\calQ$ into a $\calP$-bimodule. If one uses the monoidal definition of operads to define bimodules, that fact is a tautology: both the left action $\calP\circ\calQ\to\calQ$ and the right action $\calQ\circ\calP\to\calQ$ arise from the morphism $\phi$ and the operad structure of $\calQ$ as composites $\calP\circ\calQ\to\calQ\circ\calQ\to\calQ$ and $\calQ\circ\calP\to\calQ\circ\calQ\to\calQ$, respectively.  

\begin{theorem}
We have an isomorphism of $\Lie$ algebras 
 \[
\PL\cong \Lie\circ\calT(\CL).
 \]
More precisely, if we denote by $\calY$  the subspecies of $\PL$ spanned by all elements of the form $\ell\triangleleft\ell'+\ell'\triangleleft\ell$, where $\ell,\ell'$ are Lie monomials, the following statements hold:
\begin{enumerate}
\item we have an isomorphism $\calY\cong\CL$;
\item the suboperad $\calP$ generated by $\calY$ is free;
\item the $\Lie$-subalgebra of $\PL$ generated by $\calP$ is free and coincides with $\PL$.
\end{enumerate}
\end{theorem}

\begin{proof}
While the operad $\PL$ is defined in terms of one operation $\triangleleft$, namely as the quotient of the free operad by the operadic ideal generated by the pre-Lie identity, we shall use a different presentation, using the operations $[a_1,a_2]=a_1\triangleleft a_2-a_2\triangleleft a_1$ and $a_1\bullet a_2=a_1\triangleleft a_2+a_2\triangleleft a_1$ mentioned in the introduction. It is known \cite{MR3203367} that the $S_3$-module of identities satisfied by these operations is generated by the Jacobi identity for $[a_1,a_2]$ and the identity 
\begin{multline}\label{eq:relPL}
(a_1\bullet a_2)\bullet a_3 - a_1\bullet (a_2 \bullet a_3) - a_1\bullet [a_2, a_3] - [a_1, a_2]\bullet a_3 \\- 2[a_1, a_3]\bullet a_2+
 [a_1, a_2\bullet a_3] + [a_1\bullet a_2, a_3] + [[a_1, a_3], a_2] = 0 ,
\end{multline}

Let us start with considering the free operad $\calT([-,-],-\bullet-)$ generated by the anti-symmetric operation $[-,-]$ and the symmetric operation $-\bullet-$ (before imposing either the Jacobi identity or Identity \eqref{eq:relPL}). This operad has a suboperad $\calT([-,-])$, and therefore, we may consider it as a bimodule over that suboperad, with the action induced by the inclusion and the operad structure. The argument that follows is a bit intricate, and so we split it into eight relatively simple steps.

\textsl{Step one: factorizing elements of the free operad. }
The free operad operad $\calT([-,-],-\bullet-)$ has components indexed by finite sets $I$; those components are spanned by \emph{tree tensors}, which are rooted binary trees whose leaves are in a bijection with $I$, and whose internal vertices are labelled $[-,-]$ or $-\bullet-$ each, modulo the action of tree isomorphisms that take into account the symmetries of labels of internal vertices, so that, for example, 
 \[
\lbincomb{[-,-]}{-\bullet-}{2}{1}{3}=\lbincomb{[-,-]}{-\bullet-}{1}{2}{3}=-\rbincomb{[-,-]}{-\bullet-}{3}{1}{2} .
 \]
Let us consider a typical tree tensor $T$ in our free operad. We may consider the maximal subtree of the underlying tree that contains the root of the tree and all whose internal vertices are labelled $[-,-]$; that includes the possibility that the corresponding subtree is the trivial tree (corresponding to the unit element of the free operad). If we denote by $S$ the tree tensor corresponding to this subtree, and by $T_1,\ldots,T_p$ the tree tensors corresponding to the subtrees of the underlying tree of $T$ grafted at the leaves of $S$, in the free operad we have thus the unique factorization
 \[
T=\gamma(S;T_1,\ldots,T_p), 
 \] 
where each of the underlying trees of each tree tensor $T_i$ is either a trivial tree or a tree whose root vertex is labelled $-\bullet-$.\\

\textsl{Step two: bimodule filtration of the free operad. } Using the obtained factorization, one can define a sort of grading of the free operad. Concretely, if $T$ is a tree tensor factorized in the form $T=\gamma(S;T_1,\ldots,T_p)$ as above, we define the weight of $T$ as the sum of $p$ (arity of $S$) and the number of vertices labelled $-\bullet-$ in $T$. Note that this weight does not make our free operad a weight graded operad: it is not  additive with respect to operadic compositions. 

Let us show that the weight grading we just defined has some compatibility with the $\calT([-,-])$-bimodule structure on $\calT([-,-],-\bullet-)$. First, we see that the left action of $[-,-]$ respects the weight in a very strong sense: for $\alpha$ of weight $\ell$ and $\alpha'$ of weight $\ell'$, the element $[\alpha,\alpha']$ has weight exactly $\ell+\ell'$. Indeed, both the parameter $p$ and the number of vertices labelled $-\bullet-$ are additive in this case. Next, we observe that the right action of $[-,-]$ exhibits certain compatibility with the weight: for $\alpha$ of weight $\ell$, we have $\alpha\circ_i[-,-]$ is at least $\ell$. Indeed, our substitution may increase the parameter $p$, and does not change the number of vertices labelled $-\bullet-$. This means that if we define the decreasing filtration $F^\bullet\calT$ of the free operad with the filtered component $F^\ell\calT$ being the linear span of all tree tensors of weight at least $\ell$, this filtration is a $\calT([-,-])$-bimodule filtration. \\

\textsl{Step three: further factorization in the free operad. }
For an element $T=\gamma(S;T_1,\ldots,T_p)$ of the free operad factorized as above, each of the tree tensors $T_i$ above can be uniquely written as an iterated operadic composition of tree tensors for which \emph{only} the root is labelled $-\bullet-$, and all other internal vertices are labelled $[-,-]$; this includes the possibility of the empty composition, for which the result is the trivial tree. Indeed, to find such composition, one should locate all internal vertices of $T_i$ labelled $-\bullet-$, and cut the underlying tree of $T_i$ into maximal subtrees rooted at such vertices with all other internal vertices labelled $[-,-]$. Overall, our factorization can be summarized by the formula
 \[
\calT([-,-],-\bullet-)\cong \calT([-,-])\circ \calT(\overline{\calY}),
 \]
where $\overline{\calY}$ is the subspecies of $\calT([-,-],-\bullet-)$ spanned by all elements of the form $\alpha\bullet\alpha'$, where $\alpha,\alpha'\in \calT([-,-])$. Indeed, $\overline{\calY}$ is precisely the species spanned by tree tensors for which only the root is labelled $-\bullet-$, and all other internal vertices are labelled $[-,-]$, and taking the free operad on $\overline{\calY}$ amounts to taking all possible iterated operadic compositions. \\

\textsl{Step four: passing to the quotient. }
Let us now record what all this means for the quotient operad $\PL$. Each element of $\PL$ is a linear combination of cosets of tree tensors, each of which can be written in the way described above, that is as a composition of a coset of a tree tensor all whose internal vertices are labelled $[-,-]$ with iterated operadic compositions of cosets of tree tensors whose root is labelled $-\bullet-$ and all other internal vertices are labelled $[-,-]$. The former cosets clearly span the suboperad $\Lie\subset\PL$, and the latter cosets span the subspecies $\calY\subset\PL$. Thus, the isomorphism 
 \[
\calT([-,-],-\bullet-)\cong \calT([-,-])\circ \calT(\overline{\calY}),
 \]
induces a surjective map of left $\Lie$-modules
 \[
\pi\colon \Lie\circ\calT(\calY)\twoheadrightarrow\PL.
 \]
Let us also note that one can define a filtration $F^\bullet\PL$ of the underlying species of the operad $\PL$ analogously to the filtration of the underlying species of the free operad defined above; namely, the filtered component $F^\ell\PL$ is defined as the species of elements that can be written as linear combinations of tree tensors of weight at least $\ell$. The same argument as above (at the end of the ``step two'' part) shows that this filtration is compatible with the $\Lie$-bimodule structure on $\PL$ obtained from the morphism of operads $\Lie\to\PL$ (at this point, just existence of a linear combination of tree tensors of appropriate weight is used, so non-uniqueness coming from the fact that we imposed some relations does not create problems). Note that this filtration is \emph{not} compatible with the full operad structure; however, as the remaining part of the proof will show, compatibility with the $\Lie$-bimodule structure is sufficient to furnish a proof of our main result. This is perhaps the most subtle aspect of our argument. \\

\textsl{Step five: passing to the associated graded bimodule.}
Let us now consider the associated graded $\Lie$-bimodule $\gr_F\PL$. To make a key advance in the proof, we shall do a small calculation.  Denoting \eqref{eq:relPL} by~$r$, and computing $\frac13(r-2r.(23))$, we obtain the relation
\begin{multline*}
[a_1, a_2]\bullet a_3-a_1\bullet [a_2, a_3]=\\
\frac13\left(-(a_1\bullet a_2)\bullet a_3- a_1\bullet (a_2 \bullet a_3)+2(a_1\bullet a_3)\bullet a_2- [a_1\bullet a_2, a_3] \right.\\
\left. + [a_1,a_2\bullet a_3] +2 [a_1\bullet a_3, a_2]  + [a_1,[a_2, a_3]]
 + [[a_1, a_2], a_3]\right).
\end{multline*}
Let us understand precisely the filtration levels of the elements in this formula. First, we note that, if we consider the elements on the right, they are either compositions of the unit element with elements of the $-\bullet-$ weight two, or compositions of Lie monomials of arity two with elements of the $-\bullet-$ weight one, or, finally, Lie monomials of arity three, so all of them belong to $F^3\PL$. However, the elements whose difference we compute on the left correspond to the tree tensors 
 \[
\lbincomb{-\bullet-}{[-,-]}{1}{2}{3} \text{ and }\rbincomb{-\bullet-}{[-,-]}{1}{2}{3}
 \]
which have $-\bullet-$ as the root label. We note that in the free operad the desired decomposition $\gamma(S;T_1,\ldots,T_p)$ of each of those tree tensors is trivial: we take $S$ to be the unit element, $p=1$, and $T_1$ equal to the corresponding tree tensor. This immediately implies that these elements belong to $F^2\PL$, as the arity of the unit element and the $-\bullet-$ weight add up to $2$. It remains to check that  these tree tensors individually do not belong to $F^3\PL$ (in principle, it could be possible that imposing relations leads to a way to represent them as linear combinations of elements of larger weight). If that were true, we could write each of the three tree tensors with three leaves that have $-\bullet-$ as the root label and $[-,-]$ as the label of the non-root vertex as a linear combination of tree tensors of weight~$3$. However, this would imply that the $S_3$-module generated by \eqref{eq:relPL} is at least three-dimensional, while in reality it is of dimension two, so we have a contradiction, and our tree tensors do not belogn to $F^3\PL$.  

This implies that in the associated graded $\Lie$-bimodule, we have
 \[
[a_1, a_2]\bullet a_3=a_1\bullet [a_2, a_3].
 \]
This relation is precisely the defining relation of the cyclic Lie module. Since $\calY$ is clearly a right $\Lie$-submodule of $\PL$, and since $F^\bullet\PL$ is a filtration by $\Lie$-bimodules, we conclude that there is a surjective map of right $\Lie$-modules 
 \[
\tilde\phi\colon\CL\twoheadrightarrow \tilde\calY
 \]
where $\tilde\calY$ is the right submodule of $\gr_F\PL$ generated by the coset of the element $-\bullet-$. \\

\textsl{Step six: dimension counting and bijectivity. } Since components of these species are $\k S_n$-modules and $\k$ is of characteristic zero, passing to the associated graded does not change the module structure, our previous argument implies that there is a surjective map of right $\Lie$-modules
 \[
\phi\colon\CL\twoheadrightarrow\calY,
 \]
and hence a surjective map of linear species 
 \[
\Lie\circ\calT(\CL)\twoheadrightarrow\Lie\circ\calT(\calY)\twoheadrightarrow\PL.
 \]
However, a calculation already done in \cite[Sec.~9]{MR3091764} shows that the species $\Lie\circ\calT(\CL)$ and $\PL$ are of the same dimension in each arity, so a surjective map between the species $\Lie\circ\calT(\CL)$ and $\PL$ must be an isomorphism. Our surjection is constructed using two different surjections, and since the final result is an isomorphism, each of those surjections is an isomorphism. In particular, the surjection 
 \[
\phi\colon\CL\twoheadrightarrow\calY,
 \]
is an isomorphism, so we have a species isomorphism $\CL\cong\calY$, proving the first claim of the theorem. Furthermore, the surjective map of left $\Lie$-modules
 \[
\Lie\circ\calT(\calY)\twoheadrightarrow\PL
 \]
is an isomorphism, which implies both that the suboperad of $\PL$ generated by $\calY$ is free and the left $\Lie$-module generated by that suboperad is free and coincides with $\PL$, proving the last two claims of the theorem. 
\end{proof}

\begin{remark}
In our previous work \cite[Th.~3]{MR3203367}, an alternative proof of the result of Bergeron and Loday \cite{MR2763748} on freeness of the suboperad of $\PL$ generated by $-\bullet-$ was suggested. However, that proof contains a gap that cannot be fixed by any moderate adjustment: in fact, it is possible to show that there is no quadratic Gr\"obner basis or convergent rewriting system of the operad $\PL$ with the indicated leading terms. The proof above uses the same presentation of $\PL$ in a conceptually different way, and, in particular, furnishes a new proof of the result of Bergeron and Loday.   
\end{remark}

\section*{Funding} 
This research was supported by Institut Universitaire de France, by the University of Strasbourg Institute for Advanced Study through the Fellowship USIAS-2021-061 within the French national program ``Investment for the future'' (IdEx-Unistra), and by the French national research agency project ANR-20-CE40-0016.

\section*{Acknowledgements} I am grateful to Frédéric Chapoton for useful discussions dating back to 2010 and for comments on the first draft of this article, and to Paul Laubie whose ongoing research project made me return to this question. I also thank the anonymous referee for insisting on making the argument very precise, leaving no stone unturned.

\bibliographystyle{plain}
\bibliography{biblio}

\begin{thebibliography}{10}

\bibitem{MR2667815}
Nantel Bergeron and Muriel Livernet.
\newblock A combinatorial basis for the free {L}ie algebra of the labelled
  rooted trees.
\newblock {\em J. Lie Theory}, 20(1):3--15, 2010.

\bibitem{MR2763748}
Nantel Bergeron and Jean-Louis Loday.
\newblock The symmetric operation in a free pre-{L}ie algebra is magmatic.
\newblock {\em Proc. Amer. Math. Soc.}, 139(5):1585--1597, 2011.

\bibitem{cayley_2009}
Arthur Cayley.
\newblock {\em On the Theory of the Analytical Forms called Trees}, volume~3 of
  {\em Cambridge Library Collection - Mathematics}, page 242–246.
\newblock Cambridge University Press, 2009.

\bibitem{MR2682539}
Fr\'{e}d\'{e}ric Chapoton.
\newblock Free pre-{L}ie algebras are free as {L}ie algebras.
\newblock {\em Canad. Math. Bull.}, 53(3):425--437, 2010.

\bibitem{MR3091764}
Fr\'{e}d\'{e}ric Chapoton.
\newblock Fine structures inside the {P}re{L}ie operad.
\newblock {\em Proc. Amer. Math. Soc.}, 141(11):3723--3737, 2013.

\bibitem{MR1827084}
Fr\'{e}d\'{e}ric Chapoton and Muriel Livernet.
\newblock Pre-{L}ie algebras and the rooted trees operad.
\newblock {\em Internat. Math. Res. Notices}, (8):395--408, 2001.

\bibitem{MR3203367}
Vladimir Dotsenko.
\newblock Freeness theorems for operads via {G}r\"{o}bner bases.
\newblock In {\em O{PERADS} 2009}, volume~26 of {\em S\'{e}min. Congr.}, pages
  61--76. Soc. Math. France, Paris, 2013.

\bibitem{MR1935036}
L.~Foissy.
\newblock Finite-dimensional comodules over the {H}opf algebra of rooted trees.
\newblock {\em J. Algebra}, 255(1):89--120, 2002.

\bibitem{MR161898}
Murray Gerstenhaber.
\newblock The cohomology structure of an associative ring.
\newblock {\em Ann. of Math. (2)}, 78:267--288, 1963.

\bibitem{MR1358617}
E.~Getzler and M.~M. Kapranov.
\newblock Cyclic operads and cyclic homology.
\newblock In {\em Geometry, topology, \& physics}, Conf. Proc. Lecture Notes
  Geom. Topology, IV, pages 167--201. Int. Press, Cambridge, MA, 1995.

\bibitem{MR186708}
C.~M. Glennie.
\newblock Some identities valid in special {J}ordan algebras but not valid in
  all {J}ordan algebras.
\newblock {\em Pacific J. Math.}, 16:47--59, 1966.

\bibitem{MR1633004}
Dirk Kreimer.
\newblock On the {H}opf algebra structure of perturbative quantum field
  theories.
\newblock {\em Adv. Theor. Math. Phys.}, 2(2):303--334, 1998.

\bibitem{MR2954392}
Jean-Louis Loday and Bruno Vallette.
\newblock {\em Algebraic operads}, volume 346 of {\em Grundlehren der
  mathematischen Wissenschaften [Fundamental Principles of Mathematical
  Sciences]}.
\newblock Springer, Heidelberg, 2012.

\bibitem{MR1394355}
Alan Robinson and Sarah Whitehouse.
\newblock The tree representation of {$\Sigma_{n+1}$}.
\newblock {\em J. Pure Appl. Algebra}, 111(1-3):245--253, 1996.

\bibitem{MR0158414}
\`E.~B. Vinberg.
\newblock The theory of homogeneous convex cones.
\newblock {\em Trudy Moskov. Mat. Ob\v{s}\v{c}.}, 12:303--358, 1963.

\bibitem{Whitehouse}
Sarah~Ann Whitehouse.
\newblock {\em Gamma (co)homology of commutative algebras and some related
  representations of the symmetric group}.
\newblock PhD thesis, University of Warwick, 1994.

\end{thebibliography}

\end{document}